\documentclass{amsart}

\usepackage{amssymb}
\usepackage{amsmath}
\usepackage{cite}
\usepackage{bbm}

\title{Instability and Singularity of Projective Hypersurfaces}
\author{Cheolgyu Lee}
\address[CL]{Center for Geometry and Physics, Institute for Basic Science (IBS), Pohang 37673, Republic of Korea}
\address[CL]{Department of Mathematics, POSTECH,
77 Cheongam-ro, Nam-gu, Pohang, Gyeongbuk, 37673, Korea.
}
\email{ghost279.math@gmail.com}
\thanks{This work was supported by IBS-R003-D1.}
\date{\today}

\newtheorem{theorem}{Theorem}[section]
\newtheorem{lemma}[theorem]{Lemma}

\begin{document}
\begin{abstract}
In this paper, we will show that the Hesselink stratification of a Hilbert scheme of hypersurfaces is independent of the choice of Pl\"ucker coordinate and there is a positive relation between the length of Hesselink's worst virtual 1-parameter subgroup and multiplicity of a projective hypersurface.
\end{abstract}
\maketitle
\section{Introduction}
 Let $k$ be an algebraically closed field. The table in \cite[p. 80]{GIT} says that a hypersurface over $k$ is unstable if and only if it has a singular point of some special tangent cone for some possibly small degree and dimension. 

If there is a singular point $p$ of a hypersurface $X$, we can measure its {\it magnitude} of singularity by its multiplicity $n_{p, X}$, which is defined to be the degree of the tangent cone $TC_{p}X$ as a subscheme of the tangent space $T_{p}X$(\cite[p. 258]{Joe}). In particular, $n_{p, X}\geq 2$ if and only if $p$ is a singular point of $X$. The number $n_{X}=\max_{p\in X}n_{p, X}$ is a geometric invariant.

On the other hands, we can measure how much a hypersurface is unstable. A hypersurface of a projective space is represented by a $k$ point of a Hilbert scheme $\textup{Hilb}^{P}(\mathbb{P}^{r}_{k})$, when $P$ is of the form 
\begin{displaymath}
P(t)=\binom{r+t}{r}-\binom{r+t-d}{r}
\end{displaymath}
for some $r$ and $d$. There is a canonical action of $\textup{SL}_{r+1}(k)$ on $\textup{Hilb}^{P}(\mathbb{P}^{r}_{k})$, possibly with a choice of Pl\"ucker coordinate
\begin{displaymath}
\textrm{Hilb}^{P}(\mathbb{P}^{r}_{k})\rightarrow\mathbb{P}\bigg( \bigwedge^{Q(d+t)}H^{0}(\mathbb{P}^{r}_{k}, \mathcal{O}_{\mathbb{P}^{r}_{k}}(d+t))\bigg)
\end{displaymath}
for some $t\geq 0$. For each choice of $t$, we have the unstable locus of a Hilbert scheme and Hesselink stratification of the unstable locus
\begin{equation}
\label{Hesselink}
 \textrm{Hilb}^{P}(\mathbb{P}^{r}_{k})^{\textrm{us}}_{d+t}=\coprod_{[\lambda], \delta >0}{E_{d+t, [\lambda], \delta}}
\end{equation}
where each $E_{d+t, [\lambda], \delta}$ is a constructible subset of the Hilbert scheme(\cite{Hesselink}).

In this paper, we will show that every Hesselink stratification in \eqref{Hesselink} is actually independent of the choice of the natural number $t$(Theorem~\ref{theorem1}). We will also show that there is a positive relation between $n_{H}$ and the pair $([\lambda], \delta)$ where $x\in E_{d, [\lambda], \delta}$ for {\it the} Hilbert point $x$ which represents $H$, for every projective hypersurface $H$ over $k$(Theorem~\ref{result}). This is given by \eqref{resulteqn}, which is a sharp inequality.
\section{Preliminaries and Notations}
A sequence of closed immersions which map a Hilbert Scheme $\textrm{Hilb}^{P}(\mathbb{P}^{r}_{k})$ into projective spaces is given in \cite[(3.4)]{Gotzmann}. If $g_{P}$ is the Gotzmann number corresponding to $P$, then for any $d\geq g_{P}$, there is a closed immersion
\begin{displaymath}
\phi_{d}:\textrm{Hilb}^{P}(\mathbb{P}^{r}_{k})\rightarrow\mathbb{P}\bigg( \bigwedge^{Q(d)}H^{0}(\mathbb{P}^{r}_{k}, \mathcal{O}_{\mathbb{P}^{r}_{k}}(d))\bigg).
\end{displaymath}
Here, 
\begin{displaymath}
Q(d)=h^{0}(\mathbb{P}^{r}_{k}, \mathcal{O}_{\mathbb{P}^{r}_{k}}(d))-P(d).
\end{displaymath}
Let $S=k[x_{0}, \ldots, x_{r}]$ be graded by the degree of polynomial. Then we have natural isomorphism $S \cong \oplus_{i\geq 0} H^{0}(\mathbb{P}^{r}_{k}, \mathcal{O}_{\mathbb{P}^{r}_{k}}(i))$ as graded rings over $k$. \\
For any closed point $x\in \textrm{Hilb}^{P}(\mathbb{P}^{r}_{k})$ which represents a saturated homogeneous ideal $I\subset S$, $\phi_{d}(x)=[I^{\wedge Q(d)}]$. We can describe the set of closed points in the image of $\phi_{d}$ explicitly:
\begin{displaymath}
\phi_{d}\left( \textrm{Hilb}^{P}(\mathbb{P}^{r}_{k})\right) = \lbrace V\in \textrm{Gr} ( S_{d}  ,Q(d) )|\dim_{k}S_{1}V\leq \dim_{k}S_{1}U, \forall U\in \textrm{Gr} ( S_{d}  ,Q(d) )\rbrace.
\end{displaymath}
Here we consider $\textrm{Gr} ( S_{d}  ,Q(d) )$ as a closed subscheme of $\mathbb{P}\left( \bigwedge^{Q(d)} S_{d} \right)$ via Pl\"ucker embedding.\\
Now $\mathbb{P}\left( \bigwedge^{Q(d)} S_{d} \right)$ admits a canonical $G:=\textrm{GL}_{r+1}(k)$-action which stabilizes the image of $\phi_{d}$ so that we can discuss about the unstable locus of a Hilbert Scheme via induced $H=\textup{SL}_{r+1}(k)$ action.\cite[Theorem 1.19 and Theorem 2.1, 2nd Chpater]{GIT}. We call $\phi_{d}(x)$ a $d$'th Hilbert point for any $x\in\textrm{Hilb}^{P}(\mathbb{P}^{r}_{k})$ and $d\geq g_{P}$.
\subsection{Hesselink Stratification of a Hilbert Scheme and State Polytope}
We can define a real-valued norm $\Vert\cdot \Vert$ on the group of every 1-parameter subgroups $\Gamma(G)$ of $G$, which satisfies conditions in \cite[p.305]{Kempf}. Let $T_{0}$ be the maximal torus of $G$ which consists of all diagonal matrices. Then $\Gamma(T_{0})\cong \mathbb{Z}^{r+1}$ so there is the Euclidean norm on $\Gamma(T_{0})\otimes \mathbb{R}$, associated with the standard basis given by $\{\lambda_{i}|0\leq i\leq r\}$ satisfying
\begin{displaymath}
 \lambda_{i}(t)_{ab}=\left\lbrace\begin{array}{ll}
                     t & a=b=i\\
                     1 & a=b\neq i \\
                     0 & a\neq b\end{array} \right. .
\end{displaymath}
 For any $g\in G$ and $\lambda\in\Gamma(G)$, let $g\star\lambda\in\Gamma(G)$ be the 1-parameter subgroup which maps $t\in\mathbb{G}_{m}$ to $g\lambda(t)g^{-1}$. The norm $\Vert\cdot\Vert$ is invariant under the action of the Weyl group of $T_{0}$ in $G$.  Thus this norm can be extended to $\Gamma(G)$ via conjugation by \cite[Corollary A, p. 135]{Humphreys}. Let $T=H\cap T_{0}$, which is a maximal torus of $H$.
 
Suppose $\lambda'\in\Gamma(G)$ and $V$ be a $G$-representation. There is a decomposition
\begin{displaymath}
V=\bigoplus_{i\in\mathbb{Z}}V_{i}
\end{displaymath}
where $V_{i}=\{w\in V|\lambda'(t).w=t^{i}w\}$(\cite[Proposition 4.7.]{Mukai})and there is a decomposition $v=\sum_{i\in\mathbb{Z}}v_{i}$ with $v_{i}\in V_{i}$ so that we can explain the value $\mu(v, \lambda')$ as follows:
\begin{displaymath}
\mu(v, \lambda')=\min\{i\in\mathbb{Z}|v_{i}\neq 0\}.
\end{displaymath}
For any unstable $d$'th Hilbert point $\phi_{d}(x)$, any two indivisible worst 1-parameter subgroups of $\phi_{d}(x)$ are conjugate so that we can measure the magnitude of instability of $\phi_{d}(x)$ by the conjugacy classes of indivisible 1-parameter subgroups\cite[Theorem 3.4]{Kempf}. Let $\Lambda_{x, d}$ be the set of all indivisible worst 1-parameter subgroups of $\phi_{d}(x)$. That is, 
\begin{displaymath}
\Lambda_{x, d}=\bigg\lbrace\lambda\in\Gamma(H)\bigg\vert \lambda\textrm{ is indivisible and } \frac{\mu(\phi_{d}(x), \lambda)}{\Vert \lambda \Vert} =\max_{\lambda'\in\Gamma(H)}\frac{\mu(\phi_{d}(x), \lambda')}{\Vert \lambda' \Vert}. \bigg\rbrace.
\end{displaymath}
The set $\Lambda_{x, d}$ is non-empty. Moreover, 
\begin{displaymath}
E_{d, [\lambda], \delta}=\bigg\lbrace x\in \textrm{Hilb}^{P}(\mathbb{P}^{r}_{k})\bigg\vert \Lambda_{x, d}\cap [\lambda]\neq\emptyset,\quad \frac{\mu(\phi_{d}(x), \lambda_{\textup{max}})}{\Vert \lambda_{\textup{max}} \Vert}=\delta,\textrm{ }\forall \lambda_{\textup{max}}\in\Lambda_{x, d}. \bigg\rbrace
\end{displaymath}
is a constructible subset of $ \textrm{Hilb}^{P}(\mathbb{P}^{r}_{k}) $ for a conjugacy class $[\lambda]$ containing an indivisible 1-parameter subgroup $\lambda$ of $T$\cite{Hesselink}.
 Consequently, there is a stratification
\begin{displaymath}
 \textrm{Hilb}^{P}(\mathbb{P}^{r}_{k})^{\textrm{us}}_{d}=\coprod_{[\lambda], \delta >0}{E_{d, [\lambda], \delta}}
\end{displaymath}
for each integer $d\geq g_{P}$.

The $T_{0}$ action induced by the G action on $V_{d}^{P}=\wedge^{Q(d)} S_{d}$ has decomposition
\begin{displaymath}
V_{d}^{P}=\bigoplus_{m\in M_{dQ(d)}}V_{m, d}^{P}
\end{displaymath}
where $V_{m, d}^{P}$ is a $k$-subspace of $\wedge^{Q(d)}S_{d}$ which is spanned by
\begin{displaymath}
\bigg\lbrace\bigwedge_{i=1}^{Q(d)} m_{i}\bigg\vert m_{i}\in M_{d} \textrm{ for all }1\leq i\leq Q(d) \textrm{ and } \prod_{i=1}^{Q(d)}m_{i}=m \bigg\rbrace.
\end{displaymath}
Here $M_{d}$ is the set of monomials of degree $d$ in $S$. This decomposition coincides with weight decomposition of the $T_{0}$ action because wedges of monomials are exactly the eigenvectors of the action. Multiplicative semi-group of monomials in $S$ is isomorphic to $\mathbb{N}^{r+1}$ so that it is naturally embedded in $X(T_{0})\cong \mathbb{Z}^{r+1}$ via morphism sending an eigenvector to its weight, which is given by a character of $T_{0}$. There is also a natural isomorphism $\eta:\Gamma(T_{0})\rightarrow X(T_{0})$ defined by a basis $\{\lambda_{i}\}_{i=0}^{r}$ of $\Gamma(T_{0})$ and its dual basis with respect to the pairing in\cite[p.304]{Kempf}. For arbitrary $d$'th Hilbert point $\phi_{d}(x)$,
\begin{displaymath}
\phi_{d}(x)=\left[ \sum_{m\in M_{dQ(d)}} \phi_{d}(x)_{m} \right]
\end{displaymath}
for unique choice of sequence (up to scalar multiplication) $\{\phi_{d}(x)_{m}\}_{m\in M_{dQ(d)} }$ satisfying $\phi_{d}(x)_{m}\in V_{m, d}^{P}$. Let
\begin{displaymath}
\Xi_{x, d}=\lbrace m\in M_{dQ(d)}|\phi_{d}(x)_{m}\neq 0 \rbrace.
\end{displaymath}
This set is called the state of the Hilbert point $\phi_{d}(x)$(\cite{Kempf} and \cite{Ian}). For $f\in S_{d}$ let $\{f_{m}\}_{m\in M_{d}}$ be the sequence satisfying
\begin{displaymath}
f=\sum_{m\in M_{d}}f_{m}m.
\end{displaymath}
 Let $\Delta_{x, d}$ be the convex hull of $\Xi_{x, d}\otimes \mathbb{R}$ in $X(T_{0})\otimes\mathbb{R}\cong\mathbb{R}^{r+1}$ and $|\Delta_{x, d}|$ be the distance between $\Delta_{x, d}$ and $\xi_{d}:= \frac{dQ(d)}{r}\mathbbm{1}\in X(T_{0})\otimes\mathbb{R}$. $\mathbbm{1}\in\mathbb{R}^{r+1}$ is the vector whose coefficients are 1.  There is also a unique point $h_{x, d}$ in $\Delta_{x, d}$ satisfying $\Vert h_{x, d}-\xi_{d}\Vert =|\Delta_{x, d}|$ and unique  indivisible $\lambda_{x, d}\in\Gamma(T)$ satisfying $\eta(\lambda_{x, d})\otimes\mathbb{R}=qh_{x. d}$ for some $q\in\mathbb{R}^{+}$. Now we are ready to state
\begin{theorem}
\label{base}
Suppose $\lambda\in \Gamma(H)$ is an indivisible worst 1-parameter subgroup of $\phi_{d}(x)$ for some  $d\geq g_{P}$ and $x\in \textup{Hilb}^{P}(\mathbb{P}^{r}_{k})^{\textup{us}}_{d}$. Then,
\begin{displaymath}
\phi_{d}(x)\in E_{d, [\lambda_{g.x, d}], |\Delta_{g.x, d}|}
\end{displaymath}
For every $g\in G$ satisfying the image of $g\star\lambda$ is a subset of $T$. Such a $g\in G$ exists for any choice of $x$ and $d$. In particular, $\lambda\in[\lambda_{g.x, d}]$ and $|\Delta_{g.x, d}|=\max_{h\in G}{|\Delta_{h.x, d}|}$.
\end{theorem}
\begin{proof}
\cite[Lemma 3.2., Theorem 3.4.]{Kempf} and \cite[Criterion 3.3]{Ian}.
\end{proof}
Theorem~\ref{base} means that a problem finding $E_{d, [\lambda], \delta}$ containing $\phi_{d}(x)$ is an optimization problem. That is, it is equivalent to the problem finding $g\in G$ maximalizng $|\Delta_{g.x, d}|$.

\subsection{The Multiplicity of a Projective Hypersurface at a point}
Suppose $p=[1:0:\ldots :0]\in \mathbb{P}^{r}_{k}$ where $X$ is a projective hypersurface corresponding to the homogeneous ideal generated by the single generator $f\in S_{d}$, without loss of generality. Multiplicity of $X$ at a point $p$ is 
\begin{equation}
\label{mdef}
n_{p, X}=\min\bigg\{t\in\mathbb{N}\bigg\vert x_{0}^{d-t}|f \bigg\}.
\end{equation}
This is an extrinsic definition. However, we have also an intrinsic definition. $n_{p, X}$ is the degree of the tangent cone of $X$ at $p$ as a subscheme of the tangent space $T_{p}(X)$ of $X$ at $p$.(\cite[p. 258]{Joe})
\section{Computation of Worst State Polytope}
Let $\langle , \rangle$ be a standard inner-product on $X(T)\otimes\mathbb{R}$ with respect to the basis $\{\eta(\lambda_{i})\}_{i=0}^{r}$. For $\lambda\in\Gamma(T)$, let's define a monomial order $<_{\lambda}$ as follows:
\begin{displaymath}
m<_{\lambda}m' \Longleftrightarrow \bigg\langle \eta(\lambda)\otimes\frac{1}{\Vert\lambda\Vert}, m\bigg\rangle < \bigg\langle \eta(\lambda)\otimes\frac{1}{\Vert\lambda\Vert}, m'\bigg\rangle \textrm{ or}
\end{displaymath}
\begin{displaymath}
\left[\bigg\langle \eta(\lambda)\otimes\frac{1}{\Vert\lambda\Vert}, m\bigg\rangle = \bigg\langle \eta(\lambda)\otimes\frac{1}{\Vert\lambda\Vert}, m\bigg\rangle \textrm{ and } m<_{\textrm{lex}}m'\right].
\end{displaymath}
Here $<_{\textrm{lex}}$ is a lexicographic order with respect to the term order $x_{i}<x_{i+1}$.
We can describe $\Delta_{x, t}$ for $t\geq g_{P}$ with these notations.
\begin{equation}
\label{asymptotic}
|\Delta_{x, t}|=\max_{\lambda\in\Gamma(T)}\min_{m\in \Xi_{x, t}}\bigg\langle \eta(\lambda)\otimes\frac{1}{\Vert\lambda\Vert}, m\bigg\rangle.
\end{equation}
The value 
\begin{displaymath}
\min_{m\in \Xi_{x, t}}\bigg\langle \eta(\lambda)\otimes\frac{1}{\Vert\lambda\Vert}, m\bigg\rangle
\end{displaymath}
is equal to
\begin{displaymath}
\frac{\mu(\phi_{t}(x), \lambda)}{\Vert\lambda\Vert}.
\end{displaymath}
\subsection{Stability of Hesselink Stratifications of a Hilbert Scheme of Hypersurfaces} In this subsection,  assume that
\begin{equation}
\label{assumption}
P(t)=\binom{r+t}{r}-\binom{r+t-d}{r}.
\end{equation}
Considering a lex-segment homogeneous ideal whose Hilbert polynomial is $P$, we get $g_{P}=d$(\cite{Gotzmann}). This means that $\textup{Hilb}^{P}(\mathbb{P}^{r}_{k})$ is a parameter space of every hypersurface of $\mathbb{P}^{r}_{k}$ defined by a homogeneous polynomial of degree $d$. Also, we have a Hesselink stratification

\begin{displaymath}
 \textrm{Hilb}^{P}(\mathbb{P}^{r}_{k})^{\textrm{us}}_{t}=\coprod_{[\lambda], \delta > 0}{E_{t, [\lambda], \delta}}
\end{displaymath}
for each integer $t\geq d$. Before we state a new theorem, let 
\begin{displaymath}
\tau(\delta, D)=|M_{D}|\delta
\end{displaymath}
for $D\in \mathbb{N}$ and $\delta\in\mathbb{R}^{+}$.
\begin{theorem}
\label{theorem1}
For above stratifications, there are identities 
\begin{displaymath}
E_{d+D, [\lambda], \tau(\delta, D) }=E_{d, [\lambda], \delta}
\end{displaymath}
as constructible subsets of $\textup{Hilb}^{P}(\mathbb{P}^{r}_{k})$. Consequently, all Hesselink stratifications of  $\textup{Hilb}^{P}(\mathbb{P}^{r}_{k})^{\textrm{us}}_{t}$ corresponding to $t\geq d$ coincide. In particular, $\textup{Hilb}^{P}(\mathbb{P}^{r}_{k})^{\textup{us}}_{d+D}=\textup{Hilb}^{P}(\mathbb{P}^{r}_{k})^{\textup{us}}_{d}$ as sets.
\end{theorem}
\begin{proof}
We will show that $E_{d+D, [\lambda], \tau(\delta, D) }=E_{d, [\lambda], \delta}$ as a subset of $\textup{Hilb}^{P}(\mathbb{P}^{r}_{k})$. Suppose that $x\in E_{d, \left[\lambda\right], \delta}\cap  E_{d+D, \left[\lambda'\right], \delta'}$ for some $D\in\mathbb{N}\setminus\{0\}$, $\lambda,\lambda'\in\Gamma(T)$ and $\delta, \delta'>0$. Fix $\gamma\in\Gamma(T)$. For each $t\geq d$, we have a unique sequence $\{m_{i, t}\}_{i=1}^{|M_{t}|}$ satisfying  $m_{i, t}\in M_{t}$ for all $1\leq i\leq |M_{t}|$ and $m_{i, t}>_{\gamma}m_{i+1, t}$ for all $1\leq i<|M_{t}|$. Let $f$ be the unique generator (up to scalar) of the saturated ideal represented by $x$. Then we can write $[g.f]=g.\phi_{d}(x)$ for arbitrary $g\in G$ as follows:
\begin{displaymath}
g.f=\sum_{i=1}^{|M_{d}|}(g.f)_{i}m_{i, d}.
\end{displaymath}
Therefore,
\begin{equation}
\label{lowdeg}
\min_{m\in\Xi_{g.x, d}} \bigg\langle \eta(\gamma)\otimes\frac{1}{\Vert\gamma\Vert}, m\bigg\rangle=\bigg\langle \eta(\gamma)\otimes\frac{1}{\Vert\gamma\Vert}, m_{\alpha, d}\bigg\rangle
\end{equation}
where $\alpha=\max\{1\leq i\leq |M_{d}|\textrm{ }|(g.f)_{i}\neq 0\}$. \\
$V_{d+D}^{P}$ has a basis $\{ \wedge_{j=1}^{|M_{D}|}m_{a_{j}, d+D}|a_{j}<a_{j+1}, 1\leq a_{j}\leq |M_{D}|\}$. Now the coefficient of $\wedge_{j=1}^{|M_{D}|}m_{a_{j}, d+D}$ in $g.\phi_{d+D}(x)$ with this basis can be written as follows:
\begin{displaymath}
\left[g.\bigwedge_{i=1}^{|M_{D}|}m_{i, D}f \right]_{\wedge_{j=1}^{|M_{D}|}m_{a_{j}, d+D}}=\left[\bigwedge_{i=1}^{|M_{D}|}(g.m_{i, D})(g.f) \right]_{\wedge_{j=1}^{|M_{D}|}m_{a_{j}, d+D}}
\end{displaymath}
\begin{displaymath}
=(\det g)^{\binom{r+D}{r}}\left[\bigwedge_{i=1}^{|M_{D}|}m_{i, D}(g.f) \right]_{\wedge_{j=1}^{|M_{D}|}m_{a_{j}, d+D}}
\end{displaymath}
\begin{displaymath}
=(\det g)^{\binom{r+D}{r}}\sum_{\sigma\in P_{r+1}}\textup{sgn}(\sigma)\prod_{i=1}^{|M_{D}|}\left[m_{i, D}(g.f)\right]_{m_{a_{\sigma(i)}, d+D}}
\end{displaymath}
where $P_{r+1}$ is the permutation group on $r+1$. This means that 
\begin{displaymath}
g.\phi_{d+D}(x)=\left[\sum_{1\leq a_{1}< a_{2}<\ldots<  a_{|M_{D}|}\leq |M_{d+D}|} A_{g.f, D}^{a_{1}, \ldots, a_{|M_{D}|}} \bigwedge_{j=1}^{|M_{D}|}m_{a_{j}, d+D}\right]
\end{displaymath}
where $A_{g.f, D}^{a_{1}, \ldots, a_{|M_{D}|}}$ is a $|M_{D}|\times |M_{D}|$-minor of matrix $A_{g.f, D}$ with the choice of $a_{i}$'th columns. Here $[A_{g.f, D}]_{ij}=[m_{i, D}(g.f)]_{m_{j, d+D}}$ for all $1\leq i \leq |M_{D}|$ and $1\leq j\leq |M_{d+D}|$. If a strictly increasing sequence of integers $\{a_{j}\}_{j=1}^{|M_{D}|}$ satisfies $1\leq a_{i}\leq |M_{d+D}|$ and 

\begin{displaymath}
\prod_{i=1}^{|M_{D}|}m_{a_{i}, d+D}<_{\gamma} m_{\alpha, d}^{|M_{D}|}\prod_{i=1}^{|M_{D}|}m_{i, D},
\end{displaymath}
then $\{i\in\mathbb{Z}|1\leq i\leq |M_{D}|,\textrm{ }m_{a_{i}, d+D}<_{\gamma} m_{\alpha, d}m_{i, D}\}\neq \emptyset$. Let
\begin{displaymath}
\beta=\max\{i\in\mathbb{Z}|1\leq i\leq |M_{D}|,\textrm{ }m_{a_{i}, d+D}<_{\gamma} m_{\alpha, d}m_{i, D}\}.
\end{displaymath}
By the definition of $\alpha$, 
\begin{displaymath}
[A_{g.f, D}]_{ia_{j}}=0
\end{displaymath}
for all $1\leq i\leq \beta$ and $\beta\leq j\leq |M_{D}|$. This means that
\begin{displaymath}
A_{g.f, D}^{a_{1}, \ldots, a_{|M_{D}|}}= 0.
\end{displaymath}
Therefore, 
\begin{displaymath}
\min_{m\in\Xi_{g.x, d+D}} \bigg\langle \eta(\gamma)\otimes\frac{1}{\Vert\gamma\Vert}, m\bigg\rangle = \bigg\langle \eta(\gamma)\otimes\frac{1}{\Vert\gamma\Vert}, m_{\alpha, d}^{|M_{D}|}\prod_{i=1}^{|M_{D}|}m_{i, D} \bigg\rangle
\end{displaymath}
\begin{displaymath}
=\tau \left( \bigg\langle \eta(\gamma)\otimes\frac{1}{\Vert\gamma\Vert}, m_{\alpha, d}\bigg\rangle , D\right)=\tau \left( \min_{m\in\Xi_{g.x, d}} \bigg\langle \eta(\gamma)\otimes\frac{1}{\Vert\gamma\Vert}, m\bigg\rangle , D\right)
\end{displaymath}
by \eqref{lowdeg}. Taking a $\gamma\in\Gamma(T)$ maximalizing above values, we can verify that
\begin{displaymath}
|\Delta_{g.x,d+D}|=\tau(|\Delta_{g.x, d}|, D)
\end{displaymath}
by \eqref{asymptotic}. Taking $g\in G$ maximalizing these values, we have $[\lambda']=[\lambda]=[\lambda_{g.x, d}]$ and $\delta'=\tau(\delta, D)$ by Theorem~\ref{base}.
\end{proof}

\subsection{Optimization Preserving Property of Lower-Triangular matrix}
Once we choose a solution $g\in G$ to the optimization problem arising from Theorem~\ref{base}, we can check that $lg$ is also a solution for an arbitrary lower-triangular matrix $l$ if we assume that
\begin{displaymath}
\eta(\lambda)=(a_{0}, a_{1}, \ldots, a_{r})\in X(T_{0})\otimes\mathbb{R}\cong \mathbb{R}^{r+1}
\end{displaymath}
and $a_{i}\leq a_{i+1}$ for all $0\leq i< r$. Furthermore, $h_{g.x, d}=h_{lg.x, d}$. Actually, this is a consequence of \cite[Theorem4.2.]{Kempf}.
\begin{lemma}
\label{preserving}
Suppose $g$ satisfies
\begin{displaymath}
|\Delta_{g.x, d}|=\max_{h\in G}|\Delta_{h.x, d}|,
\end{displaymath}
\begin{displaymath}
\eta(\lambda_{g.x, d})=(a_{0}, \ldots, a_{r})\in X(T_{0})\otimes\mathbb{R}\cong\mathbb{R}^{r+1}
\end{displaymath}
and $a_{i}\leq a_{i+1}$ for all $0\leq i < r$. Then, $\lambda_{lg.x, d}=\lambda_{g.x, d}\in \Lambda_{g.x, d}$ for every lower-triangular matrix $l\in H$.
\end{lemma}
\begin{proof}
Let $\lambda=\lambda_{g.x, d}$. For every lower triangular matrix $l\in H$, 
\begin{displaymath}
l^{-1}\in P(\lambda):= \{q\in H| \exists\lim_{t\rightarrow 0}\lambda(t)q\lambda(t)^{-1}\in H\}.
\end{displaymath}
Thus,
\begin{displaymath}
\mu(\phi_{d}(lg.x), \lambda)=\mu(\phi_{d}(g.x), l^{-1}\star\lambda)=\mu(\phi_{d}(g.x), \lambda)
\end{displaymath}
by \cite[Lemma 4.2.]{Ian}. Since the norm $\Vert\cdot\Vert$ on $\Gamma(G)$ is invariant under the conjugate action, $\lambda \in \Lambda_{lg.x, d}\cap \Gamma(T)$ so that $\lambda_{lg.x, d}=\lambda_{g.x, d}$ by \cite[Theorem 4.2.b)(4)]{Kempf}.
\end{proof}

\section{A Relation between Instability and Singularity}
From now on, let's assume \eqref{assumption} for Hilbert polynomial $P$. In previous sections, we proved that Hesselink stratifications of a Hilbert Scheme $\textup{Hilb}^{P}(\mathbb{P}^{r}_{k})$ of hypersurfaces is unique in some sense. In this section, we will show that for $x\in \textup{Hilb}^{P}(\mathbb{P}^{r}_{k})^{\textup{us}}_{d}$, there is a positive relation between the multiplicity $n_{H_{x}}$ of the hypersurface $H_{x}$ represented by $x$ and the Hesselink strata $E_{d, [\lambda], \delta}$ which contains $x$. We see that the multiplicity of the hypersurface $H_{x}$ at $e=[1:0:\ldots :0]\in \mathbb{P}^{r}_{k}$ determines a supporting hyperplane of $\Delta_{x, d}$. Here we choose the homogeneous coordinate $[x_{0}:x_{1}:\ldots:x_{r}]$ of $\mathbb{P}^{r}_{k}$. 
\begin{lemma}
\label{firststep}
In the above situation,
\begin{displaymath}
\Xi_{x, d}\cap x_{0}^{d-n_{e, X}+1}M_{n_{e, X}-1}=\emptyset
\end{displaymath}
and
\begin{displaymath}
\Xi_{x, d}\cap x_{0}^{d-n_{e, X}}M_{n_{e, X}}\neq\emptyset.
\end{displaymath}

\end{lemma}
\begin{proof}
It is trivial by the definition of $\Xi_{x, d}$ and \eqref{mdef}.
\end{proof}
Lemma~\ref{firststep} means that there is a relation between a state and the multiplicity of a fixed point. We know that the set $\{\Xi_{g.x}|g\in G\}$ determines the Hesselink strata $E_{d, [\lambda], \delta}$ which contains $x$ for any unstable Hilbert point $x$ and the subgroup of $G$ generated by upper triangular matrices and permutation matrices acts on $\mathbb{P}^{r}_{k}$ transitively, whose dual action is induced by the canonical action on $V_{d}^{P}$. $n_{H_{x}}$ is a geometric invariant, so it is independent under the choice of coordinate. These facts lead us to
\begin{lemma}
\label{mainlem}
For any choice of unstable $x\in\textup{Hilb}^{P}(\mathbb{P}^{r}_{k})$, there is $g_{x}\in G$ satisfying
\begin{displaymath}
|\Delta_{g_{x}.x, d}|=\max_{h\in G}|\Delta_{h.x, d}|
\end{displaymath}
and
\begin{displaymath}
n_{H_{x}}=n_{H_{g_{x}.x}}=n_{e, H_{g_{x}.x}}.
\end{displaymath}
\end{lemma}
\begin{proof}
By Lemma~\ref{base}, there is a $g'$ satisfying 
\begin{displaymath}
|\Delta_{g'.x, d}|=\max_{h\in G}|\Delta_{h.x, d}|.
\end{displaymath}
Without loss of generality, we can guess that
\begin{displaymath}
\lambda_{g'.x, d}=(a_{0}, a_{1}, \ldots, a_{r})
\end{displaymath}
satisfies $a_{i}\leq a_{i+1}$ for all $0\leq i<r$. If it doesn't, we may take $wg'$ instead of $g'$ for some permutation matrix $w$. There is $y=[b_{0}: b_{1}: \ldots : b_{r}]\in H_{g'.x}$ satisfying $n_{H_{g'.x}}=n_{y, H_{g'.x}}$. We can choose a lower-triangular matrix $l$ such that some row of $l$ is equal to a representative vector $(b_{0}, b_{1}, \ldots , b_{r})$ of $y$. By Lemma~\ref{preserving},
\begin{displaymath}
|\Delta_{lg'.x, d}|=\max_{h\in G}|\Delta_{h.x, d}|.
\end{displaymath}
There is a permutation matrix $q$ satisfies $e.(ql)=(eq).l=y$ by the construction.
Now $g_{x}=qlg'$ has every desired property.
\end{proof}
Choosing $g_{x}\in G$ as in Lemma~\ref{mainlem} for a fixed Hilbert point $x\in E_{d, [\lambda], \delta}$, we can compare $n_{H_{x}}$ and the pair $([\lambda], \delta)$. 
\begin{theorem}
\label{result}
Suppose $x\in E_{d, [\lambda], \delta}$ for a unstable Hilbert point $x\in \textup{Hilb}^{P}(\mathbb{P}^{r}_{k})$ where
\begin{displaymath}
\lambda=(a_{0}, a_{1}, \ldots, a_{r})\in\Gamma(T)
\end{displaymath}
satisfying $\sum_{i=0}^{r}a_{i}=0$. If $b=\max_{0\leq i\leq r}a_{i}$ and $a=\min_{0\leq i\leq r}a_{i}$, then
\begin{equation}
\label{resulteqn}
\frac{\Vert\lambda\Vert\delta-ad}{b-a}\leq n_{H_{x}}\leq\frac{rd}{r+1}-\delta\frac{a}{\Vert\lambda\Vert}.
\end{equation}
\end{theorem}
\begin{proof}
By Lemma~\ref{mainlem}, we can assume that 
\begin{displaymath}
|\Delta_{x, d}|=\max_{h\in G}|\Delta_{h.x, d}|
\end{displaymath}
and
\begin{displaymath}
n_{H_{x}}=n_{e, H_{x}}.
\end{displaymath}
Without loss of generality, $\lambda\in\Lambda_{x, d}$. It it doesn't, we may choose a permutation matrix $q$ satisfying $q\star\lambda\in\Lambda_{x, d}$ and let $q\star\lambda$ be another representative of $[\lambda]$. It is possible because $x\in E_{d, [\lambda], \delta}$, by Theorem~\ref{base}. By definition, $\Delta_{x, d}$ contains 
\begin{displaymath}
h_{x, d}=\frac{d}{r+1}\mathbbm{1}+\frac{\delta}{\Vert\lambda\Vert}\eta(\lambda).
\end{displaymath}
Since $\Delta_{x, d}$ is the convex hull of $\Xi_{x, d}$, Lemma~\ref{firststep} and our assumption on $x$ imply that
\begin{displaymath}
n_{H_{x}}=n_{e, H_{x}}\leq d-\frac{d}{r+1}-\frac{\delta a}{\Vert\lambda\Vert}=\frac{rd}{r+1}-\delta\frac{a}{\Vert\lambda\Vert}
\end{displaymath}
if we consider the maximum value of the degree of $x_{0}$ in each monomial $m\in\Xi_{x, d}$. Now suppose 
\begin{displaymath}
n_{H_{x}}<\frac{\Vert\lambda\Vert\delta-ad}{b-a},
\end{displaymath}
then for all $0\leq j\leq r$ satisfying $a=a_{j}$ and for all $(b_{0}, \ldots, b_{r})\in\Delta_{x, d}$,
\begin{displaymath}
b(d-b_{j})+a(b_{j}-d)+ad\geq\sum_{i=0}^{r}a_{i}b_{i}\geq\delta\Vert\lambda\Vert
\end{displaymath}
by \eqref{asymptotic} so that
\begin{displaymath}
d-b_{j}\geq \frac{\Vert\lambda\Vert\delta-ad}{b-a} >n_{H_{x}}.
\end{displaymath}
This means that $n_{e, H_{p.x}}>n_{H_{x}}=n_{H_{p.x}}$, for the transposition matrix $p$ which permutes $0$ and $j$. This is a contradiction.
\end{proof}
If $\lambda=p\star(-r, 1, \ldots, 1)\in\Gamma(T)$ for some permutation matrix $p\in G$, then we see that
\begin{displaymath}
\frac{\Vert\lambda\Vert\delta-ad}{b-a}=\frac{rd}{r+1}-\delta\frac{a}{\Vert\lambda\Vert}
\end{displaymath}
so that $n_{H_{x}}$ must be a fixed value by \eqref{resulteqn}. We also see that 
\begin{displaymath}
\frac{\Vert\lambda\Vert\delta-ad}{b-a}> \frac{d}{r+1}.
\end{displaymath}
This implies that every unstable hypersurface of degree $d\geq r+1$ is singular. This is also a weaker version of \cite[Proposition 4.2., Chapter 3]{GIT}. \eqref{resulteqn} has been derived by the existence of certain coordinate but smoothness of $x\in\textup{Hilb}^{P}(\mathbb{P}^{r}_{k})$ requires a general property of each state polytope in $\{\Delta_{g.x, d}|g\in G, |\Delta_{g.x, d}|=\max_{h\in G}|\Delta_{h.x, d}|\}$. However, we can check that \eqref{resulteqn} is sharp. If $r=3$, $d=4$ and $\phi_{d}(x)=[x_{0}^{4}]\in\mathbb{P}(k[x_{0}, x_{1}, x_{2}, x_{3}]_{4})$ then $\lambda=(3, -1, -1, -1)\in\Lambda_{x, d}$ so that 
\begin{displaymath}
\frac{\Vert\lambda\Vert\delta-ad}{b-a}=\frac{rd}{r+1}-\delta\frac{a}{\Vert\lambda\Vert}=4=n_{H_{x}}.
\end{displaymath}
\bibliographystyle{plain}
\bibliography{Sing}
\end{document}